\documentclass[12pt]{amsart}
\usepackage{mathrsfs,pstricks,ifpdf,tikz}
\usepackage{amsfonts}
\usepackage{enumerate}
\usepackage{makecell}

\usepackage{latexsym}
\usepackage{amsmath,amssymb}
\usepackage{amsthm}
\usepackage{ifthen,calc}
\usepackage{color}
\usepackage{graphicx}
\usepackage{latexsym}
\usepackage{multirow}
\usepackage{diagbox}
\usepackage{float}
\usepackage[format=hang, margin=10pt]{caption}

\newtheorem{theorem}{Theorem}%[section]

\newtheorem{remark}[theorem]{Remark}
\newtheorem{lemma}[theorem]{Lemma}

\newcommand{\qihao}{\fontsize{7.5pt}{\baselineskip}\selectfont}%%%%%%字体大小

\newcounter{hours}
\newcounter{minutes}
\newcommand{\printtime}{
    \setcounter{hours}{\time/60}%
    \setcounter{minutes}{\time-\value{hours}*60}
    \ifthenelse{\value{hours}<10}{0}{}\thehours:%
    \ifthenelse{\value{minutes}<10}{0}{}\theminutes}

\begin{document}
\title{A note on the $4$-girth-thickness of $K_{n,n,n}$}

\author{Xia Guo}
\address{School of Mathematics, Tianjin University, 300072, Tianjin, China}
\email{guoxia@tju.edu.cn}
\author{Yan Yang}
\address{School of Mathematics, Tianjin University, 300072, Tianjin, China}
\email{yanyang@tju.edu.cn    (Corresponding author: Yan YANG)}
\thanks{This work was supported by NNSF of China under Grant  No. 11401430}

\begin{abstract}
The $4$-girth-thickness $\theta(4,G)$ of a graph $G$ is the minimum number of planar subgraphs of girth at least four whose union is $G$. In this paper, we obtain that the 4-girth-thickness of complete tripartite graph $K_{n,n,n}$ is $\big\lceil\frac{n+1}{2}\big\rceil$ except for $\theta(4,K_{1,1,1})=2$. And we also show that the $4$-girth-thickness of the complete graph $K_{10}$ is three which disprove the conjecture $\theta(4,K_{10})=4$ posed by Rubio-Montiel (Ars Math Contemp 14(2) (2018) 319).
\end{abstract}


\maketitle

\section{Introduction}
\noindent

The \textit{thickness} $\theta(G)$ of a graph $G$ is the minimum number of planar subgraphs whose union is $G$. It was defined by W.T.Tutte \cite{Tut63} in 1963. Then, the thicknesses of some graphs have been obtained when the graphs are hypercube \cite{Kl67}, complete graph \cite{AG76,BH65,Vas76}, complete bipartite graph \cite{BHM64} and some complete multipartite graphs \cite{CYa16,Yan14,Yan17}.

In 2017, Rubio-Montiel \cite{R17} define the $g$-girth-thickness $\theta(g,G)$ of a graph $G$ as the minimum number of planar subgraphs whose union is $G$ with the girth of each subgraph is at least $g$. It is a generalization of the usual thickness in which the $3$-girth-thickness $\theta(3,G)$ is the usual thickness $\theta(G)$. He also determined the $4$-girth-thickness of the complete graph $K_n$ except $K_{10}$ and he conjecture that $\theta(4,K_{10})=4$. Let $K_{n,n,n}$ denote a complete tripartite graph in which each part contains $n$ $(n\geq 1)$ vertices. In \cite{Yan17}, Yang obtained $\theta(K_{n,n,n})=\Big\lceil\frac{n+1}{3}\Big\rceil$ when $n\equiv 3$(mod $6$).

In this paper, we determine $\theta(4,K_{n,n,n})$ for all values of $n$ and we also give a decomposition of $K_{10}$ with three planar subgraphs of girth at least four, which shows $\theta(4,K_{10})=3$.

\section{the $4$-girth-thickness of $K_{n,n,n}$}
\noindent

\begin{lemma}\cite{BM08}\label{l 1}
A planar graph with $n$ vertices and girth $g$ has edges at most $\frac{g}{g-2}(n-2)$.
\end{lemma}

\begin{theorem}
The $4$-girth-thickness of $K_{n,n,n}$ is $$\theta(4,K_{n,n,n})=\Big\lceil\frac{n+1}{2}\Big\rceil$$
except for $\theta(4,K_{1,1,1})=2$.
\end{theorem}

\begin{proof}
It is trivial for $n=1$, $\theta(4,K_{1,1,1})=2$.
When $n>1$, because $|E(K_{n,n,n})|=3n^2$, $|V(K_{n,n,n})|=3n$, from Lemma \ref{l 1}, we have
$$\theta(4,K_{n,n,n})\geq\Big\lceil\frac{3n^2}{2(3n-2)}\Big\rceil=
\Big\lceil{\frac{n}{2}+\frac{1}{3}+\frac{2}{3(3n-2)}}\Big\rceil=
%\big\lceil{\frac{n}{2}+\frac{1}{3}}\big\rceil=
\Big\lceil\frac{n+1}{2}\Big\rceil
.$$

In the following, we give a decomposition of $K_{n,n,n}$ into $\big\lceil\frac{n+1}{2}\big\rceil$ planar subgraphs of girth at least four to complete the proof.
Let the vertex partition of $K_{n,n,n}$ be $(U,V,W)$, where $U=\{u_1, \dots, u_{n}\}$, $V=\{v_1, \dots,v_{n}\}$ and $W=\{w_1,\dots,w_n\}$.
In this proof, all the subscripts of vertices are taken modulo $2p$, except that of $u_{2p+1},v_{2p+1},w_{2p+1}$.

{\bf Case 1.}~ When $n=2p ~(p\geq 1)$.

Let $G_1,\dots,G_{p+1}$ be the graphs whose edge set is empty and vertex set is the same as $V(K_{2p,2p,2p})$.

\noindent {\bf Step 1:} For each $G_i$ $(1\leq i\leq p)$, arrange all the vertices $u_1,v_{3-2i},u_2,v_{4-2i},$ $u_3,v_{5-2i},\dots,u_{2p},v_{2p-2i+2}$ on a circle and join $u_j$ to $v_{j+2-2i}$ and $v_{j+1-2i}$, $1\leq j\leq 2p$. Then we get a cycle of length $4p$, denote it by $G_i^1$ $(1\leq i\leq p)$.

\noindent {\bf Step 2:} For each $G_i^1$ $(1\leq i\leq p)$, place the vertex $w_{2i-1}$ inside the cycle and join it to  $u_1,\dots,u_{2p}$, place the vertex $w_{2i}$ outside the cycle and join it to  $v_1,\dots,v_{2p}$. Then we get a planar graph $G_i^2$ $(1\leq i\leq p)$.

\noindent {\bf Step 3:} For each $G_i^2$ $(1\leq i\leq p)$, place vertices $w_{2j}$ for $1\leq j\leq p$ and $j\neq i$, inside of the quadrilateral $w_{2i-1}u_{2i-1}v_1u_{2i}$ and join each of them to vertices $u_{2i-1}$ and $u_{2i}$. Place vertices $w_{2j-1}$, for $1\leq j\leq p$ and $j\neq i$, inside of the quadrilateral $w_{2i}v_{2i-1}u_{k}v_{2i}$, in which $u_k$ is some vertex from $U$. Join each of them to vertices $v_{2i-1}$ and $v_{2i}$. Then we get a planar graph $\overline{G}_i$ $(1\leq i\leq p)$.

\noindent {\bf Step 4:} For $G_{p+1}$, join $w_{2i-1}$ to both $v_{2i-1}$ and $v_{2i}$, join $w_{2i}$ to  both $u_{2i-1}$ and $u_{2i}$, for $1\leq i\leq p$, then we get a planar graph $\overline{G}_{p+1}$.

For  $\overline{G}_1\cup \cdots \cup \overline{G}_{p+1}=K_{n,n,n}$, and the girth of $\overline{G}_{i}$ $(1\leq i\leq p+1)$ is at least four, we obtain a $4$-girth planar decomposition of $K_{2p,2p,2p}$ with $p+1$ planar subgraphs. Figure \ref{f 2} shows a $4$-girth planar decomposition of $K_{4,4,4}$ with three planar subgraphs.

\vskip0.5cm
\begin{figure}[htb]
\begin{center}
\begin{tikzpicture}
[inner sep=0pt]
\node(11) at(-3.15,1)[circle,draw]{\scriptsize $u_1$};
\node(21) at(-2.25,1)[circle,draw]{\scriptsize $v_1$};
\node(12) at(-1.35,1)[circle,draw]{\scriptsize $u_2$};
\node(22) at(-0.45,1) [circle,draw]{\scriptsize $v_2$};
\node(13) at(0.45,1)[circle,draw]{\scriptsize $u_3$};
\node(23) at(1.35,1)[circle,draw]{\scriptsize $v_3$};
\node(14) at(2.25,1)[circle,draw]{\scriptsize $u_4$};
\node(24) at(3.15,1) [circle,draw]{\scriptsize $v_4$};
\draw[-](11) to  (21);\draw[-](21) to  (12);
\draw[-](12) to  (22);\draw[-](22) to  (13);
\draw[-](13) to  (23);\draw[-](23) to  (14);
\draw[-](14) to  (24);
\draw[-](11)..controls+(0,2)and+(-1,0)..(0,3);
\draw[-](24)..controls+(0,2)and+(1,0)..(0,3);
\node(31) at(0,2.6)[circle,draw]{\qihao $w_1$};
\node(34) at(-1.1,1.7)[circle,draw]{\qihao $w_4$};
\node(32) at(0,-1)[circle,draw]{\qihao $w_2$};
\node(33) at(-0.8,0.3) [circle,draw]{\qihao $w_3$};
\draw[-](31) to  (11);\draw[-](31) to  (12);
\draw[-](31) to  (13);\draw[-](31) to  (14);
\draw[-](32) to  (21);\draw[-](32) to  (22);
\draw[-](32) to  (23);\draw[-](32) to  (24);
\draw[-](34) to  (11);\draw[-](34) to  (12);
\draw[-](33) to  (21);\draw[-](33) to  (22);
%%%%%%%%%%%%%%%%%%%%%%%%%%%%%%%%%%%%%%%%%%%%%%%%%%%%%%%%%%%%%%%%%%%%%%%%%%%%%%%%
\begin{scope}[xshift=7.4cm]
\node(11) at(-3.15,1)[circle,draw]{\scriptsize $u_1$};
\node(21) at(-2.25,1)[circle,draw]{\scriptsize $v_3$};
\node(12) at(-1.35,1)[circle,draw]{\scriptsize $u_2$};
\node(22) at(-0.45,1) [circle,draw]{\scriptsize $v_4$};
\node(13) at(0.45,1)[circle,draw]{\scriptsize $u_3$};
\node(23) at(1.35,1)[circle,draw]{\scriptsize $v_1$};
\node(14) at(2.25,1)[circle,draw]{\scriptsize $u_4$};
\node(24) at(3.15,1) [circle,draw]{\scriptsize $v_2$};
\draw[-](11) to  (21);\draw[-](21) to  (12);
\draw[-](12) to  (22);\draw[-](22) to  (13);
\draw[-](13) to  (23);\draw[-](23) to  (14);
\draw[-](14) to  (24);
\draw[-](11)..controls+(0,2)and+(-1,0)..(0,3);
\draw[-](24)..controls+(0,2)and+(1,0)..(0,3);
\node(31) at(0,2.6)[circle,draw]{\qihao $w_3$};
\node(34) at(0.7,1.7)[circle,draw]{\qihao $w_2$};
\node(32) at(0,-1)[circle,draw]{\qihao $w_4$};
\node(33) at(-0.8,0.3) [circle,draw]{\qihao $w_1$};
\draw[-](31) to  (11);\draw[-](31) to  (12);
\draw[-](31) to  (13);\draw[-](31) to  (14);
\draw[-](32) to  (21);\draw[-](32) to  (22);
\draw[-](32) to  (23);\draw[-](32) to  (24);
\draw[-](34) to  (13);\draw[-](34) to  (14);
\draw[-](33) to  (21);\draw[-](33) to  (22);
\end{scope}
\end{tikzpicture}
\vskip0.5cm    (a)  The graph $\overline{G}_1$\  \ \ \ \ \ \  \ \ \  \ \  \  \  \ \ \ \  \ \ \ \ \  \ \ \     (b)  The graph $\overline{G}_2$
\end{center}
\end{figure}
\vskip0.5cm
\begin{figure}[htb]
\begin{center}
\begin{tikzpicture}
[inner sep=0pt]
\node(11) at(3,1)[circle,draw]{\scriptsize $u_1$};
\node(21) at(0,1)[circle,draw]{\scriptsize $v_1$};
\node(12) at(3,-1)[circle,draw]{\scriptsize $u_2$};
\node(22) at(0,-1) [circle,draw]{\scriptsize $v_2$};
\node(13) at(4.5,1)[circle,draw]{\scriptsize $u_3$};
\node(23) at(1.5,1)[circle,draw]{\scriptsize $v_3$};
\node(14) at(4.5,-1)[circle,draw]{\scriptsize $u_4$};
\node(24) at(1.5,-1) [circle,draw]{\scriptsize $v_4$};
\node(31) at(0,0)[circle,draw]{\qihao $w_1$};
\node(33) at(1.5,0)[circle,draw]{\qihao $w_3$};
\node(32) at(3,0)[circle,draw]{\qihao $w_2$};
\node(34) at(4.5,0) [circle,draw]{\qihao $w_4$};
\draw[-](31) to  (21);\draw[-](31) to  (22);
\draw[-](32) to  (11);\draw[-](32) to  (12);
\draw[-](33) to  (23);\draw[-](33) to  (24);
\draw[-](34) to  (13);\draw[-](34) to  (14);
\end{tikzpicture}
\vskip0.5cm  (c)  The graph $\overline{G}_3$
\caption{A $4$-girth planar decomposition of $K_{4,4,4}$}
\label{f 2}
\end{center}
\end{figure}
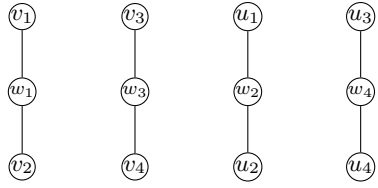

{\bf Case 2.}~ When $n=2p+1~(p > 1)$.

Base on the $4$-girth planar decomposition $\{\overline{G}_1,\cdots, \overline{G}_{p+1}\}$  of $K_{2p,2p,2p}$, by adding vertices and edges to each $\overline{G}_i$ $(1\leq i\leq p+1)$ and some other modifications on it, we will get a $4$-girth planar decomposition of $K_{2p+1,2p+1,2p+1}$ with $p+1$ subgraphs.

\noindent {\bf Step 1:} (Add $u_{2p+1}$ to $\overline{G}_{i}, 1\leq i\leq p$) For each $\overline{G}_{i}$ $(1\leq i\leq p)$,  we notice that the order of the $p-1$ interior vertices $w_{2j}$, $1\leq j\leq p,$ and $j\neq i$ in the quadrilateral $w_{2i-1}u_{2i-1}v_{1}u_{2i}$ of $\overline{G}_{i}$ has no effect on the planarity of $\overline{G}_{i}$. We adjust the order of them, such that $w_{2i-1}u_{2i-1}w_{2p-2i+2}u_{2i}$ is a face of a plane embedding of $\overline{G}_{i}$. Place the vertex $u_{2p+1}$ in this face and join it to both $w_{2i-1}$ and $w_{2p-2i+2}$. We denote the planar graph we obtain by $\widehat{G}_i$ $(1\leq i\leq p)$.

\noindent {\bf Step 2:} (Add $v_{2p+1}$ and $w_{2p+1}$ to $\widehat{G}_1$)  Delete the edge $v_1u_2$ in $\widehat{G}_1$, put both $v_{2p+1}$ and $w_{2p+1}$ in the face $w_{k}u_{1}v_{1}w_{t}v_2u_2$ in which $w_{k}$ is some vertex from $\{w_{2j}\mid 1<j\leq p\}$ and $w_{t}$ is some vertex from $\{w_{2j-1}\mid 1<j\leq p\}$. Join $v_{2p+1}$ to $w_{2p+1}$, join $v_{2p+1}$ to $u_1,u_2$, and join $w_{2p+1}$ to $v_1,v_2$, we get a planar graph $\widetilde{G}_{1}$.

\noindent {\bf Step 3:} (Add $v_{2p+1}$ and $w_{2p+1}$ to $\widehat{G}_i, 2\leq i\leq p$) For each $\widehat{G}_i$ $(2\leq i\leq p)$, place the vertex $v_{2p+1}$ in the face $w_{k}u_{2i-1}v_1u_{2i}$ in which $w_{k}$ is some vertex from $\{w_{2j}\mid 1\leq j\leq p
\mbox{ and }j\neq i\}$,  and join it to $u_{2i-1}$ and $u_{2i}$. Place the vertex $w_{2p+1}$ in the face $w_{k}v_{2i-1}u_{t}v_{2i}$ in which $w_{k}$ is some vertex from $\{w_{2j-1}\mid 1\leq j\leq p \mbox{ and }j\neq i\}$ and $u_{t}$ is some vertex from $U$. Join $w_{2p+1}$ to both $v_{2i-1}$ and $v_{2i}$, we get a planar graph $\widetilde{G}_{i}$ $(2\leq i\leq p)$.

\noindent {\bf Step 4:} (Add $u_{2p+1}, v_{2p+1}$ and $w_{2p+1}$ to $\overline{G}_{p+1}$) We add $u_{2p+1}, v_{2p+1}$ and $w_{2p+1}$ to $\overline{G}_{p+1}$. For $1\leq i\leq 2p$, join $u_{2p+1}$ to each $v_{i}$, join $v_{2p+1}$ to each $w_{i}$, join $w_{2p+1}$ to each $u_{i}$, join $u_{2p+1}$ to both $v_{2p+1}$ and $w_{2p+1}$, and join $v_{1}$ to $u_{2}$, then we get a planar graph $\widetilde{G}_{p+1}$. Figure \ref{f 3} shows a plane embedding of
$\widetilde{G}_{p+1}$.

\begin{figure}[htb]
\begin{center}
\begin{tikzpicture}
[inner sep=0pt]
\filldraw [white] (0,0) circle (1.0pt);
\begin{scope}[xshift=6.8cm]
\filldraw [black] (0,0) circle (1.0pt)
                  (-1,0) circle (1.0pt)
                  (-2,0) circle (1.0pt);
\draw[-](-2,0) to  (3,0);
\draw  (0,0.2)  node {\scriptsize $u_2$};
\draw  (-1.1,0.2)  node {\scriptsize $w_2$};
\draw  (-2,0.2)  node {\scriptsize $u_1$};
\filldraw [black] (-2.7,0) circle (1.5pt)
                  (-3,0) circle (1.5pt)
                  (-3.3,0) circle (1.5pt);
\filldraw [black] (-4,0) circle (1.0pt)
                  (-5,0) circle (1.0pt)
                  (-6,0) circle (1.0pt);
\draw[-](-4,0) to  (-6,0);
\draw  (-3.8,0.15)  node {\scriptsize $u_{2p}$};
\draw  (-5.1,0.2)  node {\scriptsize $w_{2p}$};
\draw  (-6,0.2)  node {\scriptsize $u_{2p-1}$};
%%%%%%%%%%%%%%%%%%%%%%%%%%%%%%%%%%%%%%%%%%%%%%%%%%%%%%%%%%%%%%%%%%%%%%%%%
\filldraw [black] (3,0) circle (1.0pt)
                  (1,0) circle (1.0pt)
                  (2,0) circle (1.0pt);
\draw  (1,0.2)  node {\scriptsize $v_1$};
\draw  (2,0.2)  node {\scriptsize $w_1$};
\draw  (3,0.2)  node {\scriptsize $v_2$};
\filldraw [black] (3.7,0) circle (1.5pt)
                  (4,0) circle (1.5pt)
                  (4.3,0) circle (1.5pt);
\filldraw [black] (5,0) circle (1.0pt)
                  (6,0) circle (1.0pt)
                  (7,0) circle (1.0pt);
\draw[-](5,0) to  (7,0);
\draw  (5.35,-0.2)  node {\scriptsize $v_{2p-1}$};
\draw  (6,0.2)  node {\scriptsize $w_{2p-1}$};
\draw  (6.8,0.2)  node {\scriptsize $v_{2p}$};
\filldraw [black] (0,1.5) circle (1.0pt)
                  (-3,-1.5) circle (1.0pt)
                  (4,-1.5) circle (1.0pt);
\draw  (0,1.7)  node {\scriptsize $v_{2p+1}$};
\draw  (-3,-1.7)  node {\scriptsize $w_{2p+1}$};
\draw  (4,-1.7)  node {\scriptsize $u_{2p+1}$};
\draw[-](0,1.5) to  (-1,0);\draw[-](0,1.5) to  (-5,0);
\draw[-](0,1.5) to  (2,0);\draw[-](0,1.5) to  (6,0);
\draw[-](-3,-1.5) to  (-6,0);\draw[-](-3,-1.5) to  (-4,0);
\draw[-](-3,-1.5) to  (-2,0);\draw[-](-3,-1.5) to  (0,0);
\draw[-](4,-1.5) to  (7,0);\draw[-](4,-1.5) to  (5,0);
\draw[-](4,-1.5) to  (3,0);\draw[-](4,-1.5) to  (1,0);
%%%%%%%%%%%%%%%%%%%%%%%%%%%%%%%%%%%%%%%%%%%%%%%%%%%%%%%%%%%%%%%%%
\draw  (0,1.5)..controls+(7,-0.3)and+(6,1.2)..(4,-1.5);
\draw[-](4,-1.5) to  (-3,-1.5);
\end{scope}
\end{tikzpicture}
\caption{The graph $\widetilde{G}_{p+1}$}
\label{f 3}
\end{center}
\end{figure}
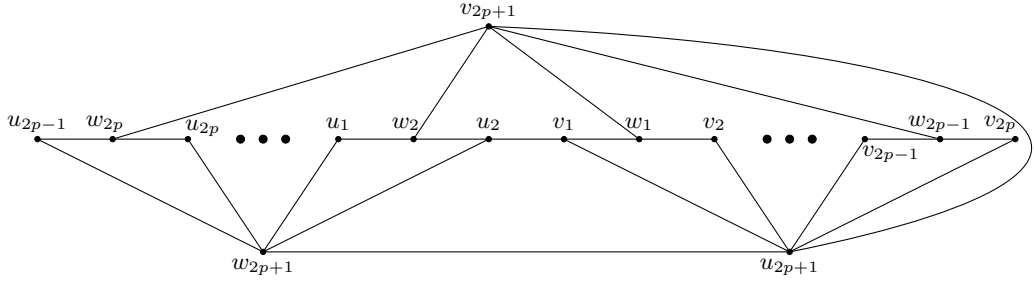

For  $\widetilde{G}_1\cup \cdots \cup \widetilde{G}_{p+1}=K_{n,n,n}$, and the girth of $\widetilde{G}_{i}$ $(1\leq i\leq p+1)$ is at least four, we obtain a $4$-girth planar decomposition of $K_{2p+1,2p+1,2p+1}$ with $p+1$ planar subgraphs. Figure \ref{f 4}  shows a $4$-girth planar decomposition of $K_{5,5,5}$ with three planar subgraphs.

\begin{figure}[htb]
\begin{center}
\begin{tikzpicture}
[inner sep=0pt]
\node(11) at(-3.15,1)[circle,draw]{\scriptsize $u_1$};
\node(21) at(-2.25,1)[circle,draw]{\scriptsize $v_1$};
\node(12) at(-1.35,1)[circle,draw]{\scriptsize $u_2$};
\node(22) at(-0.45,1) [circle,draw]{\scriptsize $v_2$};
\node(13) at(0.45,1)[circle,draw]{\scriptsize $u_3$};
\node(23) at(1.35,1)[circle,draw]{\scriptsize $v_3$};
\node(14) at(2.25,1)[circle,draw]{\scriptsize $u_4$};
\node(24) at(3.15,1) [circle,draw]{\scriptsize $v_4$};
\draw[-](11) to  (21);
\draw[-](12) to  (22);\draw[-](22) to  (13);
\draw[-](13) to  (23);\draw[-](23) to  (14);
\draw[-](14) to  (24);
\draw[-](11)..controls+(0,2)and+(-1,0)..(0,3);
\draw[-](24)..controls+(0,2)and+(1,0)..(0,3);
\node(31) at(0,2.6)[circle,draw]{\qihao $w_1$};
\node(34) at(-1.8,1.8)[circle,draw]{\qihao $w_4$};
\node(32) at(0,-1)[circle,draw]{\qihao $w_2$};
\node(33) at(-1.3,-0.4) [circle,draw]{\qihao $w_3$};
\draw(11)..controls+(0.1,1)and+(-2,-0.1).. (31);
\draw[-](31) to  (12);
\draw[-](31) to  (13);\draw[-](31) to  (14);
\draw[-](21) ..controls+(0.3,-2)and+(-1,0.1)..  (32);
\draw[-](32) to  (22);
\draw[-](32) to  (23);\draw[-](32) to  (24);
\draw[-](34) to  (11);\draw[-](34) to  (12);
\draw[-](33) to  (21);\draw[-](33) to  (22);
%%%%%%%%%%%%%%%%%%%%%%%%%%%%%%%%%%%%%%%%%%%%%%%%%%%%%%%%%%%%%%%%%%%%%%%%%%
\node(35) at(-1.3,0.2) [circle,draw]{\qihao $w_5$};
\draw[-](35) to  (21);\draw[-](35) to  (22);
\node(25) at(-2.0,1.4) [circle,draw]{\scriptsize $v_5$};
\draw[-](25) to  (11);\draw[-](25) to  (12);
\draw[-](35) to  (25);
\node(15) at(-1.2,2.1) [circle,draw]{\scriptsize $u_5$};
\draw[-](15) to  (31);\draw[-](15) to  (34);

%%%%%%%%%%%%%%%%%%%%%%%%%%%%%%%%%%%%%%%%%%%%%%%%%%%%%%%%%%%%%%%%%%%%%%%%%%%%%%%%
\begin{scope}[xshift=7.4cm]
\node(11) at(-3.15,1)[circle,draw]{\scriptsize $u_1$};
\node(21) at(-2.25,1)[circle,draw]{\scriptsize $v_3$};
\node(12) at(-1.35,1)[circle,draw]{\scriptsize $u_2$};
\node(22) at(-0.45,1) [circle,draw]{\scriptsize $v_4$};
\node(13) at(0.45,1)[circle,draw]{\scriptsize $u_3$};
\node(23) at(1.35,1)[circle,draw]{\scriptsize $v_1$};
\node(14) at(2.25,1)[circle,draw]{\scriptsize $u_4$};
\node(24) at(3.15,1) [circle,draw]{\scriptsize $v_2$};
\draw[-](11) to  (21);\draw[-](21) to  (12);
\draw[-](12) to  (22);\draw[-](22) to  (13);
\draw[-](13) to  (23);\draw[-](23) to  (14);
\draw[-](14) to  (24);
\draw[-](11)..controls+(0,2)and+(-1,0)..(0,3);
\draw[-](24)..controls+(0,2)and+(1,0)..(0,3);
\node(33) at(0,2.6)[circle,draw]{\qihao $w_3$};
\node(32) at(1.1,2)[circle,draw]{\qihao $w_2$};
\node(34) at(0,-1)[circle,draw]{\qihao $w_4$};
\node(31) at(-1.2,-0.4) [circle,draw]{\qihao $w_1$};
\draw[-](33) to  (11);\draw[-](33) to  (12);
\draw[-](33) to  (13);
\draw[-](33) ..controls+(1,-0.1)and+(0,1)..  (14);
\draw[-](21) ..controls+(0.3,-2)and+(-1,0.1)..  (34);
\draw[-](34) to  (22);
\draw[-](34) to  (23);\draw[-](34) to  (24);
\draw[-](32) to  (13);\draw[-](32) to  (14);
\draw[-](31) to  (21);\draw[-](31) to  (22);
%%%%%%%%%%%%%%%%%%%%%%%%%%%%%%%%%%%%%%%%%%%%%%%%%%%%%%%%%%%%%%%%%%%%%%%%%%%%%%%%
\node(35) at(-1.2,0.2) [circle,draw]{\qihao $w_5$};
\draw[-](35) to  (21);\draw[-](35) to  (22);
\node(25) at(1.2,1.5)[circle,draw]{\scriptsize $v_5$};
\draw[-](25) to  (13);\draw[-](25) to  (14);
\node(15) at(0.6,2.2)[circle,draw]{\scriptsize $u_5$};
\draw[-](15) to  (33);\draw[-](15) to  (32);
\end{scope}
\end{tikzpicture}
\vskip0.5cm    (a)  The graph $\widetilde{G}_1$\  \ \ \ \ \ \  \ \ \  \ \  \  \  \ \ \ \  \ \ \ \ \  \ \ \     (b)  The graph $\widetilde{G}_2$
\end{center}
\end{figure}
\begin{figure}[htb]
\begin{center}
\begin{tikzpicture}
[inner sep=0pt]
\node(13) at(-5,1)[circle,draw]{\scriptsize $u_3$};
\node(34) at(-4,1)[circle,draw]{\scriptsize $w_4$};
\node(14) at(-3,1)[circle,draw]{\scriptsize $u_4$};
\node(11) at(-2,1) [circle,draw]{\scriptsize $u_1$};
\node(32) at(-1,1)[circle,draw]{\scriptsize $w_2$};
\node(12) at(0,1)[circle,draw]{\scriptsize $u_2$};
\node(21) at(1,1)[circle,draw]{\scriptsize $v_1$};
\node(31) at(2,1) [circle,draw]{\scriptsize $w_1$};
\node(22) at(3,1)[circle,draw]{\scriptsize $v_2$};
\node(23) at(4,1)[circle,draw]{\scriptsize $v_3$};
\node(33) at(5,1) [circle,draw]{\scriptsize $w_3$};
\node(24) at(6,1) [circle,draw]{\scriptsize $v_4$};
\draw[-](13) to  (34);\draw[-](34) to  (14);
\draw[-](11) to  (32);\draw[-](32) to  (12);
\draw[-](12) to  (21);\draw[-](21) to  (31);
\draw[-](31) to  (22);\draw[-](23) to  (33);
\draw[-](33) to  (24);
\node(25) at(0.5,2.5)[circle,draw]{\scriptsize $v_5$};
\node(35) at(-2.5,-0.5) [circle,draw]{\scriptsize $w_5$};
\node(15) at(3.5,-0.5) [circle,draw]{\scriptsize $u_5$};
\draw[-](25)..controls+(6.5,-0.7)and+(4.5,1.6)..(15);
\draw[-](25) to  (34);\draw[-](25) to  (33);
\draw[-](25) to  (32);\draw[-](25) to  (31);
\draw[-](35) to  (11);\draw[-](35) to  (13);
\draw[-](35) to  (14);\draw[-](35) to  (12);
\draw[-](15) to  (21);\draw[-](15) to  (23);
\draw[-](15) to  (22);\draw[-](15) to  (24);
\draw[-](35) to  (15);
\end{tikzpicture}
\vskip0.5cm    (c)  The graph $\widetilde{G}_3$
\caption{A $4$-girth planar decomposition of $K_{5,5,5}$}
\label{f 4}
\end{center}
\end{figure}

{\bf Case 3.}~ When $n=3$, Figure \ref{f 5} shows a $4$-girth planar decomposition of $K_{3,3,3}$ with two planar subgraphs.

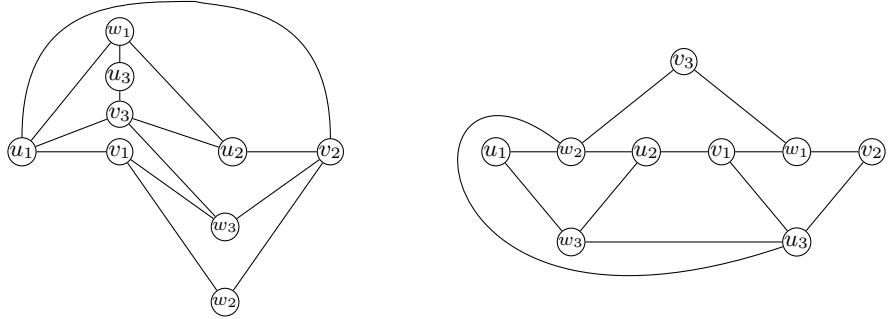
\begin{figure}[htb]
\begin{center}
\begin{tikzpicture}
[inner sep=0pt]
\node(11) at(-2.3,1)[circle,draw]{\scriptsize $u_1$};
\node(21) at(-1,1)[circle,draw]{\scriptsize $v_1$};
\node(12) at(0.5,1)[circle,draw]{\scriptsize $u_2$};
\node(22) at(1.8,1) [circle,draw]{\scriptsize $v_2$};
\draw[-](11) to  (21);
\draw[-](12) to  (22);
\draw[-](11)..controls+(0,2)and+(-1,0)..(0,3);
\draw[-](22)..controls+(0,2)and+(0.4,-0.1)..(0,3);
\node(31) at(-1.0,2.6)[circle,draw]{\qihao $w_1$};
\node(32) at(0.4,-1)[circle,draw]{\qihao $w_2$};
\draw[-](31) to  (11);\draw[-](31) to  (12);
\draw[-](32) to  (21);\draw[-](32) to  (22);
%%%%%%%%%%%%%%%%%%%%%%%%%%%%%%%%%%%%%%%%%%%%%%%%%%%%%%%%%%%%%%%%%%%%%%%%%%
\node(33) at(0.4,0.0) [circle,draw]{\qihao $w_3$};
\draw[-](33) to  (21);\draw[-](33) to  (22);
\node(23) at(-1,1.5) [circle,draw]{\scriptsize $v_3$};
\draw[-](23) to  (11);\draw[-](23) to  (12);
\node(13) at(-1,2) [circle,draw]{\scriptsize $u_3$};
\draw[-](13) to  (31);\draw[-](13) to  (23);
%%%%%%%%%%%%%%%%%%%%%%%%%%%%%%%%%%%%%%%%%%%%%%%%%%%%%%%%%%%%%%%%%%%%%%%%%%%%%%%%
\begin{scope}[xshift=6cm]
\node(11) at(-2,1) [circle,draw]{\scriptsize $u_1$};
\node(32) at(-1,1)[circle,draw]{\qihao $w_2$};
\node(12) at(0,1)[circle,draw]{\scriptsize $u_2$};
\node(21) at(1,1)[circle,draw]{\scriptsize $v_1$};
\node(31) at(2,1) [circle,draw]{\qihao $w_1$};
\node(22) at(3,1)[circle,draw]{\scriptsize $v_2$};
\draw[-](11) to  (32);\draw[-](32) to  (12);
\draw[-](12) to  (21);\draw[-](21) to  (31);
\draw[-](31) to  (22);\draw[-](23) to  (33);
\node(23) at(0.5,2.2)[circle,draw]{\scriptsize $v_3$};
\node(33) at(-1,-0.2) [circle,draw]{\qihao $w_3$};
\node(13) at(2,-0.2) [circle,draw]{\scriptsize $u_3$};
\draw[-](13)..controls+(-5.3,-1.8)and+(-2.3,1.8)..(32);
\draw[-](23) to  (32);\draw[-](23) to  (31);
\draw[-](33) to  (11);\draw[-](33) to  (12);
\draw[-](13) to  (21);\draw[-](13) to  (22);
\draw[-](33) to  (13);
\end{scope}
\end{tikzpicture}
\caption{A $4$-girth planar decomposition of $K_{3,3,3}$}
\label{f 5}
\end{center}
\end{figure}

Summarizing the above, the theorem is obtained.
\end{proof}

In \cite{R17}, the author posed the question whether $\theta(4,K_{10})=3$ or $4$, and conjectured that it is four. We disprove his conjecture by showing $\theta(4,K_{10})=3$.

\begin{remark}
The $4$-girth-thickness of $K_{10}$ is three.
\end{remark}

\begin{proof}
From \cite{R17}, we have $\theta(4,K_{10})\geq 3$. We draw a $4$-girth planar decomposition of $K_{10}$ with three planar subgraphs in Figure \ref{f 1}, which shows $\theta(4,K_{10})\leq 3$. The remark follows.

\begin{figure}[htb]
\begin{center}
\begin{tikzpicture}
[inner sep=0pt]
\node(1) at(0,1)[circle,draw]{\scriptsize $1$};
\node(2) at(1,1)[circle,draw]{\scriptsize $2$};
\node(4) at(2,1)[circle,draw]{\scriptsize $4$};
\node(3) at(3,1) [circle,draw]{\scriptsize $3$};
\node(5) at(0,0)[circle,draw]{\scriptsize $5$};
\node(6) at(1,0)[circle,draw]{\scriptsize $6$};
\node(8) at(2,0)[circle,draw]{\scriptsize $8$};
\node(7) at(3,0) [circle,draw]{\scriptsize $7$};
\node(9) at(1.5,1.8)[circle,draw]{\scriptsize $9$};
\node(10) at(1.5,-0.8)[circle,draw]{\scriptsize $10$};
\draw[-](1) to  (2);\draw[-](2) to  (4);\draw[-](4) to  (3);
\draw[-](5) to  (6);\draw[-](6) to  (8);\draw[-](8) to  (7);
\draw[-](1) to  (6);\draw[-](2) to  (8);\draw[-](4) to  (7);
\draw[-](9) to  (2);\draw[-](9) to  (3);
\draw[-](10) to  (5);\draw[-](10) to  (8);
\draw[-](9)..controls+(-1.5,-0.1)and+(-0.7,1)..(5);
\draw[-](10)..controls+(1.5,0.1)and+(0.7,-1)..(3);
%%%%%%%%%%%%%%%%%%%%%%%%%%%%%%%%%%%%%%%%%%%%%%%%%%%%%%%%%%%%%%%%%%%%%%%%%%%%%%%%
\begin{scope}[xshift=4.4cm]
\begin{scope}
\node(4) at(0,1)[circle,draw]{\scriptsize $4$};
\node(1) at(1,1)[circle,draw]{\scriptsize $1$};
\node(3) at(2,1)[circle,draw]{\scriptsize $3$};
\node(2) at(3,1) [circle,draw]{\scriptsize $2$};
\node(8) at(0,0)[circle,draw]{\scriptsize $8$};
\node(5) at(1,0)[circle,draw]{\scriptsize $5$};
\node(7) at(2,0)[circle,draw]{\scriptsize $7$};
\node(6) at(3,0) [circle,draw]{\scriptsize $6$};
\node(10) at(1.5,1.8)[circle,draw]{\scriptsize $10$};
\node(9) at(1.5,-0.8)[circle,draw]{\scriptsize $9$};
\draw[-](2) to  (3);\draw[-](3) to  (1);\draw[-](1) to  (4);
\draw[-](6) to  (7);\draw[-](7) to  (5);\draw[-](5) to  (8);
\draw[-](3) to  (6);\draw[-](1) to  (7);\draw[-](4) to  (5);
\draw[-](9) to  (7);\draw[-](9) to  (8);
\draw[-](10) to  (2);\draw[-](10) to  (1);
\draw[-](9)..controls+(-1.5,0.2)and+(-0.7,-1)..(4);
\draw[-](10)..controls+(1.5,-0.2)and+(0.7,1)..(6);
\draw[-](9)..controls+(3,0.1)and+(3,-0.1)..(10);
\end{scope}
%%%%%%%%%%%%%%%%%%%%%%%%%%%%%%%%%%%%%%%%%%%%%%%%%%%%%%%%%%%%%%%%%%%%%%%%%%%%%%%%
\begin{scope}[xshift=5.2cm]
\node(1) at(0,2)[circle,draw]{\scriptsize $1$};
\node(5) at(1,2)[circle,draw]{\scriptsize $5$};
\node(2) at(2,2)[circle,draw]{\scriptsize $2$};
\node(8) at(0,1) [circle,draw]{\scriptsize $8$};
\node(3) at(1,1)[circle,draw]{\scriptsize $3$};
\node(7) at(2,1)[circle,draw]{\scriptsize $7$};
\node(4) at(0,0)[circle,draw]{\scriptsize $4$};
\node(10) at(2,0) [circle,draw]{\scriptsize $10$};
\node(6) at(2.8,-0.8)[circle,draw]{\scriptsize $6$};
\node(9) at(-0.8,-0.8)[circle,draw]{\scriptsize $9$};
\draw[-](1) to  (5);\draw[-](5) to  (2);\draw[-](2) to  (7);
\draw[-](3) to  (7);\draw[-](3) to  (8);\draw[-](1) to  (8);
\draw[-](3) to  (5);\draw[-](8) to  (4);\draw[-](4) to  (10);
\draw[-](10) to  (7);
\draw[-](4) to (6);
\draw[-](2) to (6);
\draw[-](9) to (1);
\draw[-](9) to (6);
\end{scope}
\end{scope}
\end{tikzpicture}
\caption{A $4$-girth planar decomposition of $K_{10}$}
\label{f 1}
\end{center}
\end{figure}
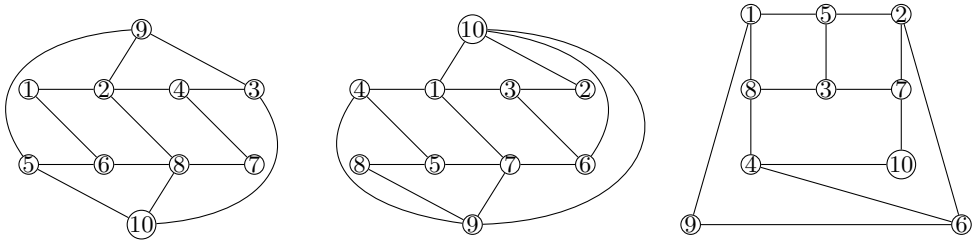

\end{proof}

\bibliography{bibfile}
\end{document}